\documentclass[a4paper,12pt]{article}
\usepackage{amssymb}
\usepackage{latexsym}
\usepackage{amsmath,amsthm,amssymb}
\usepackage{amsfonts}
\usepackage{eufrak}

\newtheorem{thm}{Theorem}[section]
\newtheorem{lem}[thm]{Lemma}
\newtheorem{pro}[thm]{Proposition}

\theoremstyle{definition}

\newtheorem{exa}[thm]{Example}
\newtheorem{rem}[thm]{Remark}

\begin{document}

\begin{center}
{\Large  Concentration of norms of random vectors with independent $p$-sub-exponential coordinates}
\end{center}
\begin{center}
{\sc Krzysztof Zajkowski\footnote{The author  declare that there is no conflict of interest.}}\\
Faculty of Mathematics, University of Bialystok \\ 
Ciolkowskiego 1M, 15-245 Bialystok, Poland \\ 
kryza@math.uwb.edu.pl 
\end{center}

\begin{abstract}
We present examples of $p$-sub-exponential random variables for any positive $p$. We prove two types of concentration of  standard $p$-norms 
($2$-norm is the Euclidean norm) of random vectors with independent $p$-sub-exponential coordinates around the Lebesgue $L^p$-norms of these $p$-norms of random vectors. In the first case $p\ge 1$, our estimates depend on the dimension $n$ of random vectors. But in the second one for $p\ge 2$, with an additional assumption, we get an estimate that does not depend on $n$. In other words, we generalize some know concentration results  in the Euclidean case to cases of the $p$-norms of random vectors with independent $p$-sub-exponential coordinates.
\end{abstract}

{\it 2020 Mathematics Subject Classification: 
60E15, 
46E30}.

{\it Key words: Concentration inequalities, Sub-Gaussian and sub-exponential random variables, Orlicz (Luxemburg) norms, 
Convex conjugates, 
Bernstein inequalities.}

\section{Introduction and results}

One of the most interesting phenomena in high-dimensional probability
is the concentration of distributions of certain functions from a large number of random variables around some deterministic value (e.g., its expected value or median, etc.); see \cite{Led} or \cite{Ver}, for instance.
In the last two decades, with the development of modern data collection, this interest has grown even more. 
Gaussian random variables are the basic tool for modeling these data.
Kahane in [10] introduced the space of sub-Gaussian random variables, which expands the possibilities of research.
It is sometimes natural to use exponential random variables and, more generally, sub-exponential variables to consider the problems under study, but more generally, it is appropriate to consider so-called, $p$-sub-exponential (sub-Weibull) random variables. 
Recently, it appears papers that take up these issues;
see  \cite{GSS} or \cite{Zang}, for example.

Our aim is to generalize the theorem on the concentration of the Orlicz norm of a vector with independent sub-Gaussian coordinates
(see Vershynin \cite[Th.3.1]{Ver}) 
to the case of the proper Orlicz norm of vector with $p$-sub-exponential coordinates.


Let $p$ be a positive number. We  consider random variables (r.v.s) with $p$-{\it sub-exponential tail decay}, i.e.,  random variables $X$ for which there exist two positive constants $c$ and $C$ such that
$$
\mathbb{P}(|X|\ge t)\le c\exp\big(-(t/C)^p\big)
$$
for all $t\ge 0$. Such random variables we will call $p$-{\it sub-exponential}.
\begin{exa}
The exponentially distributed random variable $X\sim Exp(1)$  has an exponential tail decay that is 
$\mathbb{P}(X\ge t)= \exp(-t)$.
It is the example of a random variable with a $1$-sub-exponential tail decay; $c=C=1$. 
Consider a random variable $Y_p=\theta X^{1/p}$ for some $p,\theta>0$. Observe that for $t\ge 0$ 
$$
\mathbb{P}(Y_p\ge t)=\mathbb{P}\big(\theta X^{1/p}\ge t\big)=\mathbb{P}\big(X\ge (t/\theta)^{p}\big)=\exp\big(-(t/\theta)^{p}\big).
$$
The random variable $Y_p$ has $p$-sub-exponential tail decay; $c=1$ and $C=\theta$. Let us note that $Y_p$ has the Weibull distribution with the shape parameter $p$ and the scale parameter $\theta$. One can say that random variables with Weibull distributions form model examples of r.v.s  with $p$-sub-exponential tail decay.
\end{exa}
Because it is known that random variables with the Poisson and the geometric distributions have $1$-sub-exponential tail decay then,  in a similar way as above, we can form another families of $p$-sub-exponential random variables for any $p>0$.

A more interesting case, which is independently interesting, occurs when we start with the Gaussian distribution.
\begin{exa}
Let $G$ denote a random variable with the standard normal distribution. It is  known that for tails of such variables hold the estimate: 
$$
\mathbb{P}(|G|\ge t)\le\exp(-t^2/2),
$$
for $t\ge 0$ (see for instance \cite[Prop.2.2.1]{Dudley}). Defining now $Y_p=\theta|G|^{2/p}$, by the above estimate, we get 
$$
\mathbb{P}(Y_p\ge t)=\mathbb{P}\big(\theta|G|^{2/p}\ge t\big)=\mathbb{P}\big(|G|\ge (t/\theta)^{p/2}\big)\le\exp\Big(-\big[t/(2^{1/p}\theta)\big]^p\Big).
$$
In other words, we obtain another family of r.v.s with $p$-sub-exponential tail decay; $c=1$ and $C=2^{1/p}\theta$. 
\end{exa}

Define now a symmetric random variable $G_p=\operatorname{sgn}(G)|G|^{2/p}$ ($p>0$); see \cite{Zaj1} for more details. One can calculate that its density function has the form
$$
f_{p}(x)=\frac{p}{2\sqrt{2\pi}}|x|^{p/2-1}e^{-|x|^{p}/2}.
$$
Let us emphasize that for $p=2$ we get the density of the standard normal distribution.
Observe that $\mathbb{E}G_p=0$ and $\mathbb{E}|G_p|^p=\mathbb{E}G^2=1$. 
For any $p>0$ we will call the random variable $G_p$  the {\it model $p$-normal} ($p$-{\it Gaussian}) 
and write $G_p\sim\mathcal{N}_p(0,1)$, where the first parameter denote the mean value but the second one the absolute $p$-th moment of $G_p$. 



The $p$-sub-exponential random variables can be characterized by finiteness of 
the $\psi_p$-norms defined as follows
$$
\|X\|_{\psi_p}:=\inf \big\{K>0:\; \mathbb{E}\exp(|X/K|^p)\le 2\big\};
$$
according to the standard convention $\inf\emptyset=\infty$. 
We will call the above functional $\psi_p$-norm but let us emphasize that only for $p\ge 1$ it is a proper norm. From $0<p<1$ it is the so-called quasi-norm. It does not satisfy the triangle inequality (see Appendix A in \cite{GSS} for more details). 

Let us emphasize that $p$-sub-exponential random variables $X$ satisfy the following $p$-sub-exponential tail decay:
$$
\mathbb{P}(|X|\ge t) \le 2\exp\Big(-(t/\|X\|_{\psi_p})^p\Big);
$$
see for instance \cite[Lem.2.1]{Zaj}.

For $x=(x_i)_{i=1}^n\in \mathbb{R}^n$ and $p\ge 1$, let $|x|_p$ denote the  $p$-norm of $x$, i.e., $|x|_p=(\sum_{i=1}^n|x_i|^p)^{1/p}$. For a random variable $X$, by $\|X\|_{L^p}$ we will denote the Lebesgue norm of $X$, i.e., $\|X\|_{L^p}=(\mathbb{E}|X|^p)^{1/p}$. 

From now on let $X=(X_i)_{i=1}^n$ denote a random vector with real coordinates. 
We will be interested in the concentration of the norm $|X|_p$ around $\||X|_p\|_{L^p}=(\sum_{i=1}^n\mathbb{E}|X_i|^p)^{1/p}$ in spaces of $p$-sub-exponential random variables. 
In other words, we will be interested in an estimate of the norm $\||X|_p-\||X|_p\|_{L^p}\|_{\psi_p}$.

I owe the first result of this type  to the anonymous reviewer of the previous version of this paper, to whom I hereby express my thanks.
\begin{pro}
\label{KwaP}
Let $p\ge 1$ and $X=(X_1,...,X_n)\in\mathbb{R}^n$ be a random vector with
independent $p$-sub-exponential coordinates.
Then
$$
\||X|_p-\||X|_p\|_{L^p}\|_{\psi_p}\le (2C\sqrt{n})^{1/p} K_p,
$$
where 
$K_p=\max_{1\le i\le n}\|X_i\|_{\psi_p}$ and $C$ is some universal constant. 
\end{pro}
Let us emphasize that, for $p\ge 2$, we can remove on the right hand side the factor $n^{1/(2p)}$ but under an additional assumption that $p$-th moments of the coordinates are the same, i.e., $\mathbb{E}|X_1|^p=\mathbb{E}|X_i|^p$, $i=2,3,...,n$. Let us note that then 
$\||X|_p\|_{L^p}=n^{1/p}\|X_1\|_{L^p}$.
The main theorem of this paper is the following.
\begin{thm}
\label{mthm}
Let $p\ge 2$ and $X=(X_1,...,X_n)\in\mathbb{R}^n$ be a random vector with independent $p$-sub-exponential coordinates $X_i$ that satisfy $\mathbb{E}|X_1|^p=\mathbb{E}|X_i|^p$, $i=2,3,...,n$. Then 
$$
\||X|_p-n^{1/p}\|X_1\|_{L^p}\|_{\psi_p}\le 6^{1/p}C\Big(\frac{K_p}{\|X_1\|_{L^p}}\Big)^{p-1}K_p,
$$
where $K_p:=\max_{1\le i\le n}\|X_i\|_{\psi_p}$ and $C$ is a universal constant. 
\end{thm}
\begin{rem}
The above theorem is the generalization of the concentration of $\psi_2$-norm of random vectors with independent sub-Gaussian coordinates 
(see Vershynin \cite[Th.3.1.1]{Ver}) to the case of $\psi_p$-norm of vectors
with $p$-sub-exponential coordinates, for $p\ge 2$.
\end{rem}

Norms are Lipschitz function given normed spaces. Concentration of Lipschitz functions on Gauss space is one of a basic example of concentration of measure phenomenon; see e.g. \cite[Th. 5.6]{BLM}, \cite[Cor. 3.2.6]{A-AGM} or, in a general form, \cite[Th. 5.3]{Led}. 

In Gauss space are investigated Lipschitz functions on $\mathbb{R}^n$ with respect to the Euclidean norm. In our approach, we investigate the $p$-norms. Because Orlicz spaces are Banach lattices then we can immediately formulate some form of concentration for Lipschitz functions $f$ with respect to $p$-norms of random vectors with $p$-sub-exponential coordinates, i.e.,
$$
\big\|f(X)-\|f\|_{\rm Lip}\||X|_p\|_{L^p}\big\|_{\psi_p}\le\|f\|_{\rm Lip}\big\||X|_p-\||X|_p\|_{L^p}\big\|_{\psi_p},
$$
where $\|f\|_{Lip}=\sup_{x\neq y}\frac{|f(x)-f(y)|}{|x-y|_p}$.
Let us note that in our approach a distribution of $f(X)$ does not concentrate around of a median of $f$ or a mean of $f(X)$ but around the value 
$\|f\|_{Lip}\||X|_p\|_{L^p}$. 

Before we proceed to the proofs of our results (Section \ref{sec3})  we first describe  more precisely  spaces of $p$-sub-exponential random variables (Section \ref{sec2}).

\section{Spaces of $p$-sub-exponential random variables}
\label{sec2}
The $p$-sub-exponential random variables characterize the following lemma whose proof, for $p\ge 1$, one can find \cite[Lem.2.1]{Zaj}. Let us emphasize that this proof is valid for any positive $p$.
\begin{lem}
\label{charlem}
Let $X$ be a random variable and $p> 0$. There exist positive constants $K,L,M$ 
such that 
the following conditions are equivalent:\\
 1. $\mathbb{E}\exp(|X/K|^p)\le 2$\;\;\;$(K\ge \|X\|_{\psi_p})$;\\
 2. $\mathbb{P}(|X|\ge t) \le 2\exp(-(t/L)^p)$ for all $t \ge 0$;\\ 
 3. $\mathbb{E}|X|^\alpha\le  2M^\alpha\Gamma\big(\frac{\alpha}{p}+1\big)$ for all $\alpha>0$.
\end{lem}
\begin{rem}
\label{rem1}
The definition of $\psi_p$-norm is based on condition 1. Let us notice that if condition 2 is satisfied with some constant $L$ then 
$\|X\|_{\psi_p}\le 3^{1/p}L$ (compare \cite[Rem.2.2]{Zaj}).
\end{rem}
Let $L_0$ denote the space of all random variables defined on a given probability space. By $L_{\psi_p}$ we will denote the space of random variables with finite $\psi_p$-norm:
$$
L_{\psi_p}:=\{X\in L_0:\;\|X\|_{\psi_p}<\infty\}.
$$
For $\psi_p$-norms one can formulate the following lemma (compare \cite[Corollary 4]{Zang}).
\begin{lem}
\label{norms}
Let $p,r>0$ and $X\in L_{\psi_{pr}}$ then $|X|^p\in L_{\psi_{r}}$ and 
$\||X|^p\|_{\psi_{r}}=\|X\|^p_{\psi_{pr}}$.
\end{lem}
\begin{proof}
Let $K=\|X\|_{\psi_{pr}}>0$. Then 
$$
2=\mathbb{E}\exp(|X/K|^{pr})=\mathbb{E}\exp\big(\big||X|^p/K^p\big|^{r}),
$$ 
which is equivalent to the conclusion of the lemma. 
\end{proof}
Let us emphasize that if we know the moment generating function of a given random variable $|X|$  then we can calculate the  $\psi_p$-norm of $|X|^{1/p}$.
\begin{exa}
Let $X\sim Exp(1)$. The moment generating function of $X$ equals $\mathbb{E}\exp(tX)=1/(1-t)$ for $t<1$. Let us observe that 
$$
\mathbb{E}\exp(X/K)=\frac{1}{1-1/K}\le 2
$$ 
if $K\ge 2$. It means that $\|X\|_{\psi_1}=2$. In consequence, Weibull distributed random variables, with the shape parameter $p$ and the scale parameter $\theta$,  have the $\psi_p$-norms:
$$
\|\theta X^{1/p}\|_{\psi_p}=\theta\|X\|_{\psi_1}^{1/p}=\theta 2^{1/p}.
$$

Let us note that starting with the moment generating function of $G^2$ of the form $(1-2t)^{-1/2}$ ($t<1/2$), similarly as above, one can calculate that 
$\|G^2\|_{\psi_1}=\|G\|_{\psi_2}^2=8/3$ and $\|G_p\|_{\psi_p}=(8/3)^{1/p}$.
\end{exa}
Let us notice that by Jensen's inequality we get, for $p\ge 1$, that the $\psi_p$-norm of the expected value of $p$-sub-exponential random variable is not less than the $\psi_p$-norm of this random variable itself, since
$$
2=\mathbb{E}\exp\Big(\big|X/\|X\|_{\psi_p}\big|^p\Big)\ge\exp\Big(\big|\mathbb{E}X/\|X\|_{\psi_p}\big|^p\Big)
=\mathbb{E}\exp\Big(\big|\mathbb{E}X/\|X\|_{\psi_p}\big|^p\Big),
$$
which means that $\|\mathbb{E}X\|_{\psi_p}\le \|X\|_{\psi_p}$. In consequence, for  $p$-sub-exponential random variable, we have
\begin{equation}
\label{cent}
\|X-\mathbb{E}X\|_{\psi_p}\le 2\|X\|_{\psi_p}\quad (p\ge 1).
\end{equation}

$1$-sub-exponential (simply sub-exponential) random variables  will play a special role in our considerations. Sub-exponential random variable $X$ with mean zero can be defined by finiteness of $\tau_{\varphi_1}$-norm, i.e.,
$$
\tau_{\varphi_1}(X)=\inf\{K>0:\;\ln\mathbb{E}\exp(tX)\le\varphi_\infty(Kt)\}<\infty;
$$
where $\varphi_\infty(x)=x^2/2$ for $|x|\le 1$ and $\varphi_\infty(x)=\infty$ otherwise;
see the definition of $\tau_{\varphi_p}$-norm in \cite{Zaj}, compare Vershynin \cite[Prop.2.7.1]{Ver}. Let us emphasize that the norms $\|\cdot\|_{\psi_1}$ and 
$\tau_{\varphi_1}(\cdot)$ are equivalent in the space of centered sub-exponential random variables (compare \cite[Th.2.7]{Zaj}). 
\begin{exa}
\label{przy1}
If $X$ is a exponentially distributed random variable with the parameter $1$ 
then $\mathbb{E}X=1$.
Let us note that the cumulant generating function of $X-\mathbb{E}X=X-1$ equals 
$
\ln\mathbb{E}\exp(X-1)=-t-\ln(1-t).
$
Since $C_{X-1}(0)=0$ and $C_{X-1}^\prime(0)=0$, by the Taylor formula, we get
\begin{equation}
\label{tau}
C_{X-1}(t)=\frac{1}{2}C^{\prime\prime}_{X-1}(\theta_tt)t^2\quad (|t|<1)
\end{equation}
for some $\theta_t\in(0,1)$. Let us notice that $C_{X-1}^{\prime\prime}(t)=1/(1-t)^2$ and it is an increasing function for $|t|<1$. Let us observe now that  
$\varphi_\infty(Kt)=K^2t^2/2$ if $|t|\le 1/K$ and $\infty$ otherwise.  By (\ref{tau}) we have that the infimum $K$ such that 
$C_{X-1}(t)\le \varphi_\infty(Kt)$ satisfied the equation $C_{X-1}^{\prime\prime}(1/K)=K^2$. This means that 
$$
\frac{1}{(1-1/K)^2}=K^2
$$
is the $\tau_{\varphi_1}$-norm of $(X-1)$. Solving this equation, we get 
$\tau_{\varphi_1}(X-1)=2$.

\end{exa}
In the following lemma, it is shown that sub-exponential random variables possess the {\it approximate rotation invariance property}.
\begin{lem}
Let $X_1,...,X_n$ be independent sub-exponential random variables. Then
$$
\tau_{\varphi_1}^2\Big(\sum_{i=1}^n(X_i-\mathbb{E}X_i)\Big)\le \sum_{i=1}^n\tau_{\varphi_1}^2(X_i-\mathbb{E}X_i).
$$
\end{lem}
\begin{proof}
Denote $\tau_{\varphi_1}(X_i-\mathbb{E}X_i)$ by $K_i$, $i=1,...n$. For independent centered sub-exponential r.v.s we have
\begin{eqnarray}
\label{estinf}
\mathbb{E}\exp\Big(t\sum_{i=1}^n(X_i-\mathbb{E}X_i)\Big) & = & \prod_{i=1}^n\mathbb{E}\exp\big(t(X_i-\mathbb{E}X_i)\big)\nonumber\\
\; & \le & \prod_{i=1}^n\exp \varphi_\infty(K_it)=\exp\Big(\sum_{i=1}^n\varphi_\infty(K_it)\Big).
\end{eqnarray}
Observe that
$$
\sum_{i=1}^n\varphi_\infty(K_it)=
\left\{
\begin{array}{ccl}
\frac{1}{2}(\sum_{i=1}^nK_i^2)t^2 &  if & t\le 1/\max_iK_i,\\
\infty &  otherwise. & \;
\end{array}
\right.
$$
Since $\max_iK_i\le \sqrt{\sum_{i=1}^nK_i^2}$, we get
$$
\sum_{i=1}^n\varphi_\infty(K_it)\le \varphi_\infty\Big(\Big(\sum_{i=1}^nK_i^2\Big)^{1/2}t\Big).
$$
By the above, the estimate (\ref{estinf}) and the definition of $\tau_{\varphi_1}$-norm we obtain that 
$$
\tau_{\varphi_1}\Big(\sum_{i=1}^n(X_i-\mathbb{E}X_i)\Big)\le \Big(\sum_{i=1}^n\tau_{\varphi_1}^2(X_i-\mathbb{E}X_i)\Big)^{1/2}.
$$
\end{proof}
\begin{rem}
Let us note that if $X_i$, $i =1,\ldots, n$, are sub-exponential then $|X_i|$ are sub-exponential  too.  
The above lemma implies that
\begin{equation}
\label{first}
\tau_{\varphi_1}\Big(\sum_{i=1}^n|X_i|-\sum_{i=1}^n\mathbb{E}|X_i|\Big)=\tau_{\varphi_1}\Big(|X|_1-\||X|_1\|_1\Big)
\le\sqrt{n}\max_{1\le i\le n}\tau_{\varphi_1}\Big(|X_i|-\mathbb{E}|X_i|\Big)
\end{equation}
\end{rem}
In the following example it is shown that the factor $\sqrt{n}$ on the right hand side is necessary
\begin{exa}
Let $X_i\sim Exp(1)$, $i=1,...,n$, be independent random variables. Note that the cumulant generating function of their centered sum equals $nC_{X-1}$
($X\sim Exp(1)$), i.e.,
$$
\ln\mathbb{E}\exp[t(\sum_{i=1}^nX_i-n)]=nC_{X-1}(t)=-nt-n\ln(1-t).
$$
As in Example \ref{przy1} we get
$$
nC_{X-1}(t)=\frac{n}{2}C^{\prime\prime}_{X-1}(\theta_tt)t^2\quad (|t|<1)
$$
and the $\tau_{\varphi_1}$-norm of the centered sum of $X_i$ equals $\sqrt{n}+1\sim \sqrt{n}$.
\end{exa}



By the estimate \cite[(6)]{Zaj} (i.e., $\tau_{\varphi_1}(X)\le 2\sqrt{2}\|X\|_{\psi_1}$) and the last inequality  in the proof of 
\cite[Th.2.7]{Zaj} \big(i.e., $\|X\|_{\psi_p}\le 3^{1/p}L_p \tau_{\varphi_p}(X)$\big), taking $p=1$, we get that 
$$
\frac{1}{6}\|X\|_{\psi_1}\le \tau_{\varphi_1}(X)\le 2\sqrt{2}\|X\|_{\psi_1}.
$$
Since $\||X_i|-\mathbb{E}|X_i|\|_{\psi_1}\le 2 \|X_i\|_{\psi_1}$, $i=1,...,n$,  we can rewrite the inequality (\ref{first}) to the form 
\begin{equation}
\label{first1}
\||X|_1-\||X|_1\|_{L^1}\|_{\psi_1}\le 2C\sqrt{n}K_1,
\end{equation}
where $K_1=\max_{1\le i\le n}\|X_i\|_{\psi_1}$ and $C$ is a universal constant whose infimum is less than or equal to $12\sqrt{2}$.
 
The proof of the following proposition is similar to the proof of the upper bound in the large deviation theory (see for instance 
\cite[5.11(4)Theorem. Large deviation]{GrimStir}) but with one difference. Instead of the cumulant generating function of a given random variable, we use its upper estimate by the function $\varphi_\infty$ and, in consequence, the convex conjugate $\varphi_\infty^\ast=\varphi_1$ on its tail estimate (see \cite[Lem. 2.6]{Zaj}), where 
$$
\varphi_1(x)=
\left\{
\begin{array}{ccl}
\frac{1}{2}x^2 & {\rm if} & |x|\le 1,\\
|x|-\frac{1}{2} & {\rm if} & |x|>1.
\end{array}
\right.
$$
\begin{pro}
\label{Bern}
Let $X_i$, $i=1,...,n$, be independent  sub-exponential random variables.  
Then
$$
\mathbb{P}\Big(\Big|\frac{1}{n}\sum_{i=1}^n (X_i-\mathbb{E}X_i)\Big|\ge t\Big)\le 2\exp\Big(-n\varphi_1\Big(\frac{t}{2C_1K}\Big)\Big),
$$
where $K=\max_{1\le i\le n}\|X_i\|_{\psi_1}$ 
and $C_1$ is the universal constant 
such that\\ $\tau_{\varphi_1}(\cdot)\le C_1\|\cdot\|_{\psi_1}$.
\end{pro}
\begin{rem}
From now on the constant $C_1=\inf\{C>0:\;\tau_{\varphi_1}(X)\le C_1\|X\|_{\psi_1}\;{\rm for\;any}\; X\in L_{\psi_1}\}$ is the same in each occurrence. 
\end{rem}
\begin{proof}
The moment generating function of $\frac{1}{n}\sum_{i=1}^n X_i$ can be estimated as follows
\begin{eqnarray*}
\mathbb{E}\exp\Big(u\frac{1}{n}\sum_{i=1}^n (X_i-\mathbb{E}X_i)\Big)
\;&=&\prod_{i=1}^n\mathbb{E}\exp\Big(u\frac{1}{n} (X_i-\mathbb{E}X_i)\Big)\\
\;
&\le & \prod_{i=1}^n\exp\Big(\varphi_\infty \Big(\frac{1}{n}\tau_{\varphi_1}((X_i-\mathbb{E}X_i))u\Big)\Big).
\end{eqnarray*}
By \cite[(6)]{Zaj} we have that
the right hand side can be estimate as follows
$$
 \prod_{i=1}^n\exp\Big(\varphi_\infty \Big(\frac{1}{n}C_1\| X_i-\mathbb{E}X_i\|_{\psi_1}u\Big)\Big)
\le\exp\Big(n\varphi_\infty \Big(\frac{2}{n}C_1Ku\Big)\Big).
$$

The convex conjugate of the function $f(u):=n\varphi_\infty(\frac{2}{n}C_1Ku)$ equals 
\begin{eqnarray*}
f^\ast(t)&=&\sup_{u\in\mathbb{R}}\Big\{tu-n\varphi_\infty\Big(\frac{2}{n}C_1Ku\Big)\Big\}=\sup_{u>0}\Big\{tu-n\varphi_\infty
\Big(\frac{2}{n}C_1Ku\Big)\Big\}\\
\; &=& n\sup_{u>0}\Big\{\frac{t}{2C_1K}\frac{2C_1Ku}{n}-\varphi_\infty\Big(\frac{2}{n}C_1Ku\Big)\Big\}=
n\sup_{v>0}\Big\{\frac{t}{2C_1K}v-\varphi_\infty(v)\Big\}\\
\; & = & n\varphi_1(t/2C_1K); 
\end{eqnarray*}
the second equality holds since $\varphi_\infty$ is the even function, the fourth one by the substituting $v=\frac{2}{n}C_1Ku$ and the last one by definition of the convex conjugate for even functions and the equality $\varphi_\infty^\ast=\varphi_1$.
Thus, we get
$
f^\ast(t)=n\varphi_1(t/2C_1K).
$
Similarly as in \cite[Lem. 2.4.3]{BulKoz} (formally $f$ and $f^\ast$ are not $N$-function, but the proof is the same also for these functions), 
we get
$$
\mathbb{P}\Big(\Big|\frac{1}{n}\sum_{i=1}^n (X_i-\mathbb{E}X_i)\Big|\ge t\Big)\le 2\exp\Big(-n\varphi_1\Big(\frac{t}{2C_1K}\Big)\Big).
$$
\end{proof}
\begin{rem}
Let us emphasize that because 
$$
\varphi_1\big(t/(2C_1K)\big)\ge \frac{1}{2}\min\big\{t^2/(4C_1^2K^2),t/(2C_1K)\big\}, 
$$ 
then the above estimate implies a form of Bernstein's inequality for averages
$$
\mathbb{P}\Big(\Big|\frac{1}{n}\sum_{i=1}^n (X_i-\mathbb{E}X_i)\Big|\ge t\Big)\le 2\exp\Big(-\frac{n}{2}\min\Big\{\frac{t^2}{4C_1^2K^2},\frac{t}{2C_1K}\Big\}\Big);
$$
compare Vershynin \cite[Cor.2.8.3]{Ver}.
\end{rem}
\section{Proofs of the results} 
\label{sec3}

{\it Proof of Proposition \ref{KwaP}.}
Because, for $a\ge 0$, the function $a^{1/p}$ is concave on the nonnegative half-line of real numbers, then the following inequality 
\begin{equation}
\label{triK}
\big|a-b\big|\ge\big|a^{1/p}-b^{1/p}\big|^p 
\end{equation} 
holds for any $a,b\ge 0$. 

If $X_i$, $i=1,...,n$, are $p$-sub-exponential random variables then $|X_i|^p$ are the sub-exponential ones. 
Let $Y_i$ denotes $|X_i|^p$ and $Y$ be a vector $(Y_i)_{i=1}^n$. By Lemma \ref{norms} we have $\|Y_i\|_{\psi_1}=\|X_i\|_{\psi_p}^p$. Moreover $|Y|_1=|X|_p^p$ and 
$\||Y|_1\|_{L^1}=\||X|_p\|_{L^p}^p$. Substituting in (\ref{first1}) $Y$ instead of $X$ we get
$$
\||Y|_1-\||Y|_1\|_{L^1}\|_{\psi_1}=\||X|_p^p-\||X|_p\|_{L^p}^p\|_{\psi_1}\le 2C\sqrt{n}K_p^p,
$$
where $K_p:=\max_{1\le i \le n}\|X_i\|_{\psi_p}$ and $C\le 12\sqrt{2}$ is the universal constant that appeared in (\ref{first1}).

By the definition of $\psi_1$-norm and inequality (\ref{triK}) with  $a=|X|_p^p$ and $b=
\||X|_p\|_{L^p}^p$ we obtain
$$
2\ge\mathbb{E}\exp\Big(\frac{\big||X|_p^p-\||X|_p\|_{L^p}^p\big|}{2C\sqrt{n}K_p^p}\Big)\ge
\mathbb{E}\exp\Big(\frac{\big||X|_p-\||X|_p\|_{L^p}\big|^{p}}{\big[(2C\sqrt{n})^{1/p}K_p\big]^p}\Big),
$$
which means that 
$$
\big\||X|_p-\||X|_p\|_{L^p}\big\|_{\psi_p}\le (2C\sqrt{n})^{1/p}K_p. 
$$
It finishes the proof of Proposition \ref{KwaP}.

The structure of the proof of Theorem \ref{mthm}  is similar to the proof in Vershynin \cite[Th. 3.1.1]{Ver} but, apart from  Proposition \ref{Bern} and Lemma \ref{charlem}, we also use the following two technical lemmas.
\begin{lem}
\label{xalfa}
Let $x,\delta\ge 0$ and $p\ge 1$. If $|x-1|\ge \delta$ then $|x^p-1|\ge\max\{\delta,\delta^p\}$. 
\end{lem}
\begin{proof}
Under the above assumption on $x$ and $p$ we have: $|x^p-1|\ge |x-1|$. It means that if $|x-1|\ge\delta$ then $|x^p-1|\ge\delta$. For $0\le\delta\le 1$ we have $\delta^p\le\delta$. In consequence 
$|x^p-1|\ge\max\{\delta,\delta^p\}$ for $0\le\delta\le 1$. 

Suppose now that $\delta>1$. The condition $|x-1|\ge \delta$ is equivalent to 
$x\ge\delta+1$ if $x\ge 1$ or $x\le 1-\delta$ if $0\le x\le 1$. Let us observe that the second opportunity is not possible for $\delta>1$ and $x\ge 0$. The first one gives $x^p\ge (\delta+1)^p\ge\delta^p+1$ ($p\ge 1$) that is equivalent to $x^p-1\ge\delta^p$ for $x\ge 1$. Summing up we get
$|x^p-1|\ge\max\{\delta,\delta^p\}$ for $x,\delta\ge 0$ and $p\ge 1$.
\end{proof}
\begin{lem}
\label{lem2}
If $p\ge 2$ then $\varphi_1(\max\{\gamma,\gamma^p\})\ge \frac{1}{2}\gamma^p$ for $\gamma\ge 0$.
\end{lem}
\begin{proof}
By the definition of $\varphi_1$ we have
$$
\varphi_1(\max\{\gamma,\gamma^p\})=
\left\{
\begin{array}{ccl}
\frac{1}{2}\gamma^2 & {\rm if} & 0\le\gamma\le 1,\\
\gamma^p-\frac{1}{2} & {\rm if} & 1<\gamma.
\end{array}
\right.
$$
If $0\le\gamma\le 1$ then $\varphi_1(\max\{\gamma,\gamma^p\})=\frac{1}{2}\gamma^2\ge \frac{1}{2}\gamma^p$ for $p\ge 2$.\\
If $1<\gamma$ then the inequality $\varphi_1(\max\{\gamma,\gamma^p\})=\gamma^p-\frac{1}{2}>\frac{1}{2}\gamma^p$ also holds.
\end{proof}
{\it Proof of Theorem \ref{mthm}}.
Let us observe that the expression
$$
\frac{1}{n\|X_1\|_{L^p}^p}|X|_p^p-1=\frac{1}{n}\sum_{i=1}^n\Big(\frac{|X_i|^p}{\|X_1\|_{L^p}^p}-1\Big)
$$
is the sum of independent and centered sub-exponential random variables. Moreover, by condition (\ref{cent}) and Lemma \ref{norms}, we have
$$
\||X_i|^p-1\|_{\psi_1}\le 2\||X_i|^p\|_{\psi_1}=2\|X_i\|^p_{\psi_p}\le 2K^p_p.
$$
Now, by virtue of Lemma \ref{xalfa} and Proposition \ref{Bern}, we get
\begin{eqnarray}
\label{est1}
\mathbb{P}\Big(\Big|\frac{1}{n^{1/p}\|X_1\|_{L^p}}|X|_p-1\Big|\ge \delta\Big)&\le&
\mathbb{P}\Big(\Big|\frac{1}{n\|X_1\|_{L^p}^p}|X|_p^p-1\Big|\ge \max\{\delta,\delta^p\}\Big)\nonumber\\
\; &\le& 2\exp\Big(-n\varphi_1\Big(\frac{\|X_1\|_{L^p}^p\max\{\delta,\delta^p\}}{2C_1 K^p_p}\Big)\Big)\nonumber\\
\; & \le & 2\exp\Big(-n\varphi_1\Big(\frac{\|X_1\|_{L^p}^p\max\{\delta,\delta^p\}}{C K^p_p}\Big)\Big)
\end{eqnarray}
for any $C\ge 2C_1$.

The inequality
$$
2=\mathbb{E}\exp\Big(\frac{|X_i|}{\|X_i\|_{\psi_p}}\Big)^p\ge 1+\mathbb{E}\Big(\frac{|X_i|}{\|X_i\|_{\psi_p}}\Big)^p
$$
implies that $\|X_i\|_{\psi_p}^p\ge \mathbb{E}|X_i|^p=\|X_1\|_{L_p}^p$, $i=1,...,n$, and, in consequence, 
$K_p\ge \|X_1\|_{L_p}$. 
Under this condition we have
\begin{eqnarray*}
\frac{\|X_1\|_{L^p}^p\max\{\delta,\delta^p\}}{C K^p_p}
&=& \frac{1}{C}\max\Big\{\frac{\|X_1\|_{L^p}^p\delta}{K_p^p},\frac{\|X_1\|_{L^p}^p\delta^p}{K_p^p}\Big\}\\
\; &\ge & \frac{1}{C}\max\Big\{\frac{\|X_1\|_{L^p}^p\delta}{K_p^p},\Big(\frac{\|X_1\|_{L^p}^p\delta}{K_p^p}\Big)^p\Big\}
\; ({\rm since}\;\|X_1\|_{L_p}/ K_p\le 1)\\
\; &\ge& \max\Big\{\frac{\|X_1\|_{L^p}^p\delta}{CK_p^p},\Big(\frac{\|X_1\|_{L^p}^p\delta}{CK_p^p}\Big)^p\Big\} \;({\rm assuming}\; C\ge 1 ).
\end{eqnarray*}
By the definition of $\varphi_1$ and Lemma \ref{lem2} with $\gamma= \|X_1\|_{L^p}^p\delta/(CK_p^p)$ we get
$$
\varphi_1\Big(\frac{\|X_1\|_{L^p}^p\max\{\delta,\delta^p\}}{C K^p_p}\Big)\ge \varphi_1\Big(\max\Big\{\frac{\|X_1\|_{L^p}^p\delta}{CK_p^p},\Big(\frac{\|X_1\|_{L^p}^p\delta}{CK_p^p}\Big)^p\Big\}\Big)\ge\frac{1}{2}
\Big(\frac{\|X_1\|_{L^p}^p\delta}{CK_p^p}\Big)^p.
$$
Rearranging (\ref{est1}) and applying the above estimate, we obtain the following 
\begin{eqnarray*}
\mathbb{P}\Big(\Big||X|_p-n^{1/p}\|X_1\|_{L^p}\Big|\ge n^{1/p}\|X_1\|_{L^p}\delta\Big)
&=& \mathbb{P}\Big(\Big|\frac{1}{n^{1/p}\|X_1\|_{L^p}}|X|_p-1\Big|\ge \delta\Big)\\
\; &\le& \mathbb{P}\Big(\Big|\frac{1}{n^{1/p}\|X_1\|_{L^p}}|X|_p-1\Big|\ge \delta\Big)\\
\; &= & 2\exp\Big(-\Big(\frac{n^{1/p}\|X_1\|_{L^p}^p\delta}{2^{1/p}CK_p^p}\Big)^p\Big).
\end{eqnarray*}
Changing variables to $t=n^{1/p}\|X_1\|_{L^p}\delta$, we get the following $p$-sub-exponential tail decay
$$
\mathbb{P}\Big(\Big||X|_p-n^{1/p}\|X_1\|_{L^p}\Big|\ge t\Big)\le 2\exp\Big(-\Big(\frac{\|X_1\|_{L^p}^{p-1} t}{2^{1/p}CK_p^p}\Big)^p\Big).
$$
By Lemma \ref{charlem} and Remark \ref{rem1} we obtain
$$
\big\||X|_p-n^{1/p}\|X_1\|_{L^p}\big\|_{\psi_p}\le 6^{1/p}C\Big(\frac{K_p}{\|X_1\|_{L^p}}\Big)^{p-1}K_p,\quad{\rm for}\; p\ge 2,
$$
where $C\ge\max\{2C_1,1\}$ is the universal constant.
It finishes the proof of Theorem \ref{mthm}.
\begin{exa}
Let 
$
{\bf G}_p=(G_{p,1},...,G_{p,n})
$
be a random vector with independent standard $p$-normal coordinates ( $G_{p,i}\sim\mathcal{N}_p(0,1)$).
Recall that $\|G_{p,i}\|_{L^p}=1$ and $\|G_{p,i}\|_{\psi_p}= (8/3)^{1/p}$, for $i=1,...,n$. Thus $K_p^p=8/3$. By 
Theorem \ref{mthm} we get
$$
\||{\bf G}_p|_p-n^{1/p}\|_{\psi_p}
\le \frac{8}{3}C6^{1/p}\quad{\rm for}\; p\ge 2.
$$
\end{exa}
\begin{rem}
Many problems deal with sub-Gaussian and sub-exponential random variables may be considered in the spaces of $p$-sub-exponential random variables for any positive $p$. In the paper G\"otze et al. \cite{GSS} one can find generalizations and applications of some concentration inequalities for polynomials of such variables in cases of $0<p\le 1$. In our paper, we focus our attention on  concentrations of norms of random vectors with independent $p$-sub-exponential coordinates. 
\end{rem}


\begin{thebibliography}{                    }



\bibitem{A-AGM}
Artstein-Avidan, S., Giannopoulos, A., Milman, V.D.:  Asymptotic Geometric Analysis. Part I. vol. 202, Mathematical Surveys and Monographs. American Mathematical Society, Providence, RI, (2015)

\bibitem{BLM}
 Boucheron,,, S.,  Lugosi G.,  Massart P.:  Concentration Inequalities: A Nonasymptotic Theory of Independence. Oxford University Press (2013)


\bibitem{BulKoz}
 Buldygin, V.,  Kozachenko, Yu.:  Metric Characterization of Random Variables and Random Processes, Translations of Mathematical Monographs, vol. 188. American Mathematical Society, Providence, RI (2000)


\bibitem{Dudley}
 Dudley, R.M.: {\it Uniform Central Limit Theorems}, Cambridge University Press, (1999)

\bibitem{GSS}
 G\"otze, F.,  Sambale, H.,  Sinulis, A.: {\it Concentration inequalities for polynomials in $\alpha$-sub-exponential random variables}, Electron. J. Probab. \textbf{26}, Paper No. 48, 22 pp. (2021)

\bibitem{GrimStir}
 Grimmett, G. R.,   Stirzaker, D. R.:  Probability and Random Processes, Oxford University Press, Third edition (2001)

\bibitem{Led}
 Ledoux, M.: The Concentration of Measure Phenomenon. Mathematical Surveys and Monographs, 89. American Mathematical Soc., (2001)


\bibitem{Kah}
 Kahane, J.P.: Propriétés locales des fonctions à séries de Fourier aléatoires (French). Stud. Math., {\bf 19} (no. 1), 1-25 (1960).

\bibitem{Ver} 
 Vershynin, R.: High-Dimensional Probability. Cambridge University  Press (2018) 

\bibitem{Zaj}
Zajkowski, K.:   On norms in some class of exponential type Orlicz spaces of random variables, 
Positivity {\bf 24}, 1231-1240 (2020).

\bibitem{Zaj1}
 Zajkowski, K.:  Multivariate $\alpha$-normal distributions. arXiv:2108.00272v4 (2023)

\bibitem{Zang}
Zhang, H.,  Wei, H.: Sharper sub-Weibull concentrations, Mathematics, {\bf 10}(13), 2252 (2022)

\end{thebibliography}
\end{document}